\documentclass[draft,oneside,reqno]{amsart}

\usepackage{amsmath, amsthm, amsfonts,cite}
\usepackage{wrapfig}
\usepackage{amsfonts}
\usepackage{appendix}
\usepackage{amsmath}
\usepackage{amssymb}
\usepackage{stackrel}
\usepackage{color}

\newtheorem{theorem}{Theorem}[section]

\newtheorem{prop}[theorem]{Proposition}

\newtheorem*{remark}{Remark}

\newcommand{\lt}{\left}
\newcommand{\rt}{\right}
\newcommand{\bpm}{\begin{pmatrix}}
\newcommand{\epm}{\end{pmatrix}}
\newcommand{\bsm}{\lt(\begin{smallmatrix}}
\newcommand{\esm}{\end{smallmatrix}\rt)}
\newcommand{\beq}{\begin{equation}}
\newcommand{\eeq}{\end{equation}}
\newcommand{\bal}{\begin{align}}

\renewcommand{\d}{\mathrm{d}}

\newcommand{\du}{\mathrm{d}u}
\newcommand{\dx}{\mathrm{d}x}
\newcommand{\dy}{\frac{\mathrm{d}y}{y}}
\newcommand{\dz}{\frac{\mathrm{d}x \mathrm{d}y}{y^2}}

\newcommand{\Z}{\ensuremath{\mathbb{Z}}}

\newcommand{\Q}{\ensuremath{\mathbb{Q}}}
\newcommand{\R}{\ensuremath{\mathbb{R}}}
\newcommand{\C}{\ensuremath{\mathbb{C}}}

\renewcommand{\H}{\ensuremath{\mathbb{H}}}
\newcommand{\mt}{\mathbb}

\newcommand{\h}{\ensuremath{\mathfrak{h}}}

\newcommand{\vep}{\varepsilon}
\newcommand{\ep}{\epsilon}
\newcommand{\mf}{\mathfrak}
\newcommand{\re}{\mf{Re}\,}

\DeclareMathOperator{\SL}{SL}
\DeclareMathOperator{\Res}{Res}

\newcommand{\aquad}{\qquad\qquad}
\newcommand{\bquad}{\aquad\aquad}
\newcommand{\cquad}{\bquad\bquad}

\renewcommand{\th}{\textsuperscript{th}}
\newcommand{\inv}{^{-1}}
\newcommand{\hf}{\frac{1}{2}}
\newcommand{\qtr}{\frac{1}{4}}

\renewcommand{\Re}{\operatorname{Re}}
\renewcommand{\Im}{\operatorname{Im}}

\newcommand{\kt}{\frac{k}{2}}
\newcommand{\tkt}{\tfrac{k}{2}}
\newcommand{\thf}{\tfrac{1}{2}}
\newcommand{\mc}{\mathcal}

\usepackage[usenames,dvipsnames]{xcolor}

\usepackage[normalem]{ulem}
\title{Counting Square Discriminants}
\author{Thomas A. Hulse}
\thanks{Research of the first author was partially supported by a Coleman Postdoctoral Fellowship at Queen's University.}
\address{Thomas Hulse \newline 
\indent Colby College}
\email{tahulse@colby.edu}
\author{Chan Ieong Kuan}
\address{Chan Ieong Kuan \newline 
\indent University of Maine}
\email{chan.i.kuan@maine.edu}
\author{Eren Mehmet K\i ral}
\address{Eren Mehmet K\i{}ral \newline
\indent Texas A{}\&{}M University}
\email{ekiral@tamu.edu}
\author{Li-Mei Lim}
\address{Li-Mei Lim \newline 
\indent Bard College at Simon's Rock}
\email{llim@simons-rock.edu}
\date{\today}
\begin{document}
\maketitle
\begin{abstract}
Counting integral binary quadratic forms with certain restrictions is a classical problem. In this paper, we count binary quadratic forms of fixed discriminant given restrictions on the size of their coefficients. We accomplish this by investigating the analytic properties of a certain double Dirichlet series, which is a shifted convolution sum of certain classical automorphic forms.
\end{abstract}

\section{Introduction}

The aim of this paper is to answer an elementary question concerning the asymptotics of the number of integral binary quadratic forms with fixed positive discriminant. Specifically, we consider the discriminant $b^2 -4ac$ and look at the number of integer triples $(a,b,c)$ for which $b^2 - 4ac = h>0$, where $a,b,c$ are bounded in a box of size $X \gg h$ and $h$ is fixed. 

We see we have the relation
\beq
 \#\{(a,b,c) \in \Z_{>0}^3 : a,c \leq X \text{ and } b^2-4ac=h\} =\frac{1}{2} \sum_{a,c=1}^X \tau(4ac + h)  \label{targetSum3}
\eeq
where $\tau(n)$ is the square-indicator function,
\[
\tau(n) = \begin{cases} 1 &\text{if } n =0\\ 2&\text{if } n \text{ is a nonzero square} \\ 0 & \text{if } n \text{ is not a square}. \end{cases}
\]
Our first result concerns an asymptotic estimate of a smooth variant of this sum.
\begin{theorem}\label{boxtheorem}
For large $X \gg 1$, we have that
\[
\sum_{a,c=1}^\infty \tau(4ac + h)e^{-\left( \frac{a+c}{X}\right)}  = C_1(h)X\log X + C_2(h)X +O(X^{\frac{1}{2}+\vep}),
\]
where $C_1(h)$ and $C_2(h)$ are computable constants. Furthermore, if $h$ is a square then $C_1(h)\neq0$, otherwise $C_1(h)=0$ and $C_2(h)\neq 0$. 
\end{theorem}
\begin{remark}
	The factor $e^{-(a + c)/X}$ can be exchanged for $\psi_X(a)\psi_X(c)$ where $\psi_X$ is the characteristic function on $[1,X]$. Thus we obtain a sharp cutoff result described in \eqref{targetSum3} with the same main term and error term $O(X^{1-\delta})$ for $\delta>0$. Modifying the inverse Mellin integral in \eqref{mellin}, we can get $\delta = 5/39$. 
\end{remark}

The most similar result, which motivated the work in this paper, is due to Oh-Shah \cite{oh2011limits} who proved following theorem.
\begin{theorem}[Oh-Shah\cite{oh2011limits}]
For any non-zero square $h \in \Z$ there exists $C >0$ such that
\begin{align*}
&\# \{Q(x,y) = ax^2 + bxy + cy^2: a,b,c \in \Z, \operatorname{disc} (Q) = h,  a^2 + b^2 + c^2 \leq X^2\} \\
&\cquad \aquad \quad=  C X \log X + O(X (\log X)^{\frac 34}).
\end{align*}
\end{theorem}
This result is obtained via a general ergodic technique. We note that our result has a second main term and a much improved error term. However, recent preliminary work of Oh and Shah produces an asymptotic more analogous to our smoothed sum in Theorem \ref{boxtheorem}.\cite{heemail}

Furthermore, there is a long history of study for asymptotics where the coefficients $a,c$ are bounded by the condition $ac\leq X^2$. In these cases we are investigating sums of the form
\beq
 \#\{(a,b,c) \in \Z_{>0}^3 : ac \leq X^2 \text{ and } b^2-4ac=h\} =\frac{1}{2} \mkern-15mu \sum_{\substack{a,c \\ 1 \leq ac \leq X^2}} \mkern-15mu \tau(4ac + h).  \label{targetSum2}
\eeq
When we count points in this hyperbolic-shaped region instead of the box shaped region in Theorem \ref{boxtheorem}, we get 

\begin{theorem}\label{hyperbolictheorem} For any non-zero $h \in \mathbb{Z}_{>0}$ and $X \gg h$,
\[
\sum_{\substack{a,c \\ 1 \leq ac \leq X^2}} \mkern-15mu \tau(4ac + h) = B_1(h)X(\log X)^2 + B_2(h)X\log X +B_3(h)X+O(X^{\frac{4}{5}+\vep}).
\]
where $B_1(h),B_2(h),B_3(h) $ are computable constants. Furthermore, $B_1(h)\neq0$ if $h$ is a square, and $B_1(h)=0$ and $B_2(h)\neq 0$ if $h$ is not a square.

\end{theorem}

These sums differ slightly from those in the literature, as more frequently sums of the form
\[
\sum_{b \leq X} \sigma_0(b^2-h) = \# \left\{b^2-ac=h \ | \ a,b,c \in \mathbb{Z}_{\geq 1}, \ ac \leq X^2-h \right\} 
\]
are investigated. For example, Bykovski\u{\i}\cite{bykovskii} proves the following theorem.
\begin{theorem}[Bykovski\u{\i}\cite{bykovskii}] For $h \in \mathbb{Z}_{<0}$ there exist constants $K_1(h),K_2(h)$ such that for $X \gg 1$, 
\[
\sum_{b \leq X} \sigma_0(b^2-h) = K_1(h)X\log X + K_2(h)X +O((X\log X)^{\frac{2}{3}}).
\]
and $K_1(h)$ and $K_2(h)$ have known formulas.
\end{theorem}
The above result is notable because this is best known error term. However, it is limited to the case where $h$ is negative and thus not square. There is a more general result due to Hooley\cite{hooley}:
\begin{theorem}[Hooley \cite{hooley}] For $h \in \mathbb{Z}$ not a square and $X \gg h$,
\[
\sum_{b \leq X} \sigma_0(b^2-h) = K_1(h)X\log X + K_2(h)X +O(X^{\frac{8}{9}+\vep}).
\]
\end{theorem}
This puts positive and negative $h$ on equal footing but does not deal with the case where $h$ is a square.  

We should also note that works due to Duke, Rudnick and Sarnak\cite{DRS} and Selberg \cite{LaxPhil} provide a general method that may produce results similar to the above estimates with improved error terms, but these applications do not appear with such specificity in the literature. As far as we know, our method is the only one that treats with the case where $h$ is and is not a square simultaneously and gives counts in both the hyperbolic and box-shaped geometries. While we do not investigate the case where $h$ is negative in this paper, our technique is the same and we can obtain similar results. 

Our approach to this problem uses the analytical properties of the double Dirichlet series,
\begin{equation}\label{initialDirichletSeries}
\sum_{a,c=1}^\infty \frac{\tau(4ac + h)}{a^s c^{w}},
\end{equation}
where $\tau$ is defined as before. This Dirichlet series can be obtained from the Petersson inner product
\begin{equation}\label{PVinnerProduct}
\langle P_{h,Y}^{-\hf}(*,s,\delta), \Im (*)^\qtr \overline{\theta E^*\!\lt(4*,\tfrac{1+v}{2}\rt) }\rangle
\end{equation}
 taken over the domain $\Gamma_0(4)\backslash \H$. The function $\overline{\theta(z)E^*\!\lt(4z,\tfrac{1+v}{2}\rt) (\Im z)^\qtr}$ will be denoted $\widetilde{V}_v(z)$ for brevity.  The  function $ P_{h,Y}^{-\hf}(z,s,\delta)$ is a Poincar\'{e} series, often just denoted $P$, and is inspired by a heuristic form given by
\begin{equation}\label{heuristicPoincare}
P_{h}^{-\hf}(z,s) = \sum_{\gamma \in \Gamma_\infty \backslash \Gamma_0(4)} \Im(\gamma z)^s e^{- 2\pi i h\gamma z} \frac{j(\gamma,z)}{|j(\gamma,z)|}.
\end{equation}
The letters $Y$ and $\delta$ appearing in $P_{h,Y}^{-\hf}(z,s,\delta)$ are auxiliary variables that will ensure good behavior as we describe below. This is not the standard real-analytic Poincar\'{e} series due to Selberg \cite{Selberg}, which we could use for $h$ negative, but analogous to a modified version originally constructed in \cite{Jeff} to deal with the case where $h$ is positive. 

The function $\widetilde{V}_v$ is an automorphic form of half integral weight on the space $\Gamma_0(4)\backslash \H$. In order to understand the analytic behavior of \eqref{initialDirichletSeries}, we use the spectral decomposition of the Poincar\'{e} series. This yields
\begin{align}
\langle P,\widetilde{V}_v\rangle &= \sum_{j} \langle P, u_j \rangle \langle u_j, \widetilde{V}_v\rangle + \langle P, \Im(*)^\qtr \overline{\theta}\rangle \langle \Im(*)^\qtr \overline{\theta}, \widetilde{V}_v\rangle \notag  \\
&\aquad\quad+ \sum_{\mathfrak{a} \text{ cusps}}\frac{\mathcal{V}}{4\pi}\int_{-\infty}^\infty \langle P, E(* ,\thf+it) \rangle \langle E_\mathfrak{a}(*,\thf+it), \widetilde{V}_v \rangle \  \d t. \label{poincareSpectral}
\end{align}
Here $\mathcal{V}$ is the volume of the hyperbolic surface $\Gamma_0(4) \backslash \H$.
By considering each term in the above expansion separately, we are able to find the locations of the poles and the residues of the meromorphic continuation of \eqref{initialDirichletSeries}.

The difficulties are threefold, two of which were overcome in \cite{Jeff}, where a similar Poincar\'{e} series of weight zero was defined and studied. Firstly, $P_h^{-\hf}$, as given in \eqref{heuristicPoincare}, is not an $L^2$ function and has exponential growth in $y$. Therefore, the spectral expansion does not exist, and we must work with a truncated function $P_{h,Y}(z,s)$ supported on a compact domain. After taking the inner product, we can take the limit 
\beq \label{innerp}
\lim_{Y\to \infty } \langle P_{h,Y}(*,s), \widetilde{V}_v\rangle 
\eeq
to recover the Dirichlet series. 

The second difficulty is that, although $P_{h,Y}(z,s)$ can be given a spectral decomposition, quantities such as $\langle P_{h,Y}(*,s), u_j\rangle$ do not make sense as $Y\to \infty$. That difficulty is overcome by introducing another auxiliary variable, $\delta>0$, and so the final form of the Poincar\'{e} series is given below in \eqref{poincare}. After letting $Y \to \infty$, convergence of the Dirichlet series occurs for $s$ in some right half-plane and in this region,  and so $\delta$ can be meaningfully sent to zero. Similarly, on the spectral side the expansion \eqref{poincareSpectral} converges when $\Re(s)$ sufficiently negative as $\delta \to 0$.  

The third difficulty does not have to do with the analytic complications of the Poincar\'{e} series, but is due to the fact that $\widetilde{V}_v$ is of moderate growth and thus not square-integrable. We overcome this by subtracting a modular form of the same weight, character and growth as $\widetilde{V}_v$. More specifically, we subtract linear combinations of Eisenstein series at various cusps, evaluated at specific values of the holomorphic variable such that the resulting difference is square-integrable.
 
In later sections of the paper, we use inverse Mellin transforms to get from the Dirichlet series to the truncated sum, and then move the lines of integration to obtain the desired asymptotic results. Since the methods for obtaining Theorems \ref{boxtheorem} and \ref{hyperbolictheorem} are very conceptually similar, save for some small differences, this paper will only show all of the computations used for obtaining the asymptotic of the smoothed sum in Theorem \ref{boxtheorem} and then give brief explanations for how to modify the argument to obtain our other estimates. 

%
%

\section*{Acknowledgments}
Most of this work was completed while the authors were at Brown University. We would like to thank Jeffrey Hoffstein at Brown University for bringing this problem to our attention and suggesting how we might approach it, as well as Min Lee at the University of Bristol for many helpful conversations. We would also like to thank Gergely Harcos at the Alfr\'{e}d R\'{e}nyi Institute of Mathematics for directing us toward the argument concerning the non-existence of exceptional eigenvalues given in Section \ref{spectral}.


\section{The Automorphic Function $V(z)$} \label{automorphic}

For notational convenience, let $\Gamma = \Gamma_0(4)$, the set of matrices in $\SL(2,\Z)$ which are upper triangular when reduced modulo four. Modular forms of weight $k$, for $k$ a half-integer, are functions $f$ on the upper half plane satisfying 
\[
f(\gamma z) = j(\gamma , z)^{2k} f(z),
\]
where $\gamma = \lt(\begin{smallmatrix} a&b\\c&d \end{smallmatrix}\rt) \in \Gamma$ and 
\[
j(\gamma,z) = \vep_d\inv \lt(\frac{c}{d}\rt) (cz+d)^\hf
\] 
is the weight $\hf$ cocycle, as in \cite{Shimura}.

Given a cusp $\mathfrak{a} \in \Q \cup \{\infty\}$, we denote the stabilizer of the cusp,
\[
\Gamma_\mathfrak{a} := \{\gamma \in \Gamma : \gamma \mathfrak{a} = \mathfrak{a}\}.
\]
For the group $\Gamma$, any cusp is equivalent to one of the three inequivalent cusps: $0, \hf,$ and $\infty$.

Let $E(z,w)$ be the standard real-analytic Eisenstein series on $\SL(2,\Z)$, and 
\[
E^*(z,w) := E(z,w) \zeta^*(2w)
\]
where $\zeta^*(w)$ is the completed Riemann zeta function,
\beq
\zeta^*(w)=\pi^{-\frac{w}{2}}\Gamma(\tfrac{w}{2})\zeta(w).
\eeq
Furthermore, for $k$ a half integer, let $E_{\mathfrak{a}}^k(z,w)$ be the weight $k$ Eisenstein series at the cusp $\mathfrak{a}$, as in \cite{goldfeld1985eisenstein}. Our notation differs somewhat from the notation of \cite{goldfeld1985eisenstein}, as we shift the position of the complex variable by replacing the $w$ with $w - \kt$. This amounts to using the normalized cocycle $j(\gamma,z)/|j(\gamma,z)|$ in the definition of the Eisenstein series instead of $j(\gamma,z)$ alone. Also, note that in \cite{goldfeld1985eisenstein} the weight is denoted by $\kt$ but throughout this work we use $k$.  

 Finally, recall that the classical theta series, $$\theta(z) = \sum_{n\in \Z} e^{2\pi in^2 z}= \sum_{n=0}^\infty \tau(n) e^{2\pi in z},$$
 is a modular form of weight $\hf$ and that its coefficients are the square indicator function. 

With this background in mind, we begin by fixing a positive integer $h$.
The multiple Dirichlet series,
\beq
\sum_{a,c=1}^\infty \frac{\tau(4ac + h)}{a^{s+v} c^{s}}  \label{dseries},
\eeq
 which is obtained via an inverse Mellin transform, will aid us in estimating the desired asymptotics in Theorems \ref{boxtheorem} and \ref{hyperbolictheorem}. 
Letting $m = ac$, we can rewrite the above Dirichlet series as 
\[
\sum_{m=1}^\infty  \frac{\tau(4m+h)}{m^s}\sum_{a|m} \frac{1}{a^v} = \sum_{m=1}^\infty \frac{\tau(4m+h)\sigma_{-v}(m)}{m^s}.
\]
We will use the Fourier coefficients of Eisenstein series of level 1 and weight 0 on $\SL(2,\Z)\backslash \H$ as a source for the divisor function, $\sigma_{-v}(m)$. Thus what we have is a shifted convolution sum of the Fourier coefficients of the theta function and those of the Eisenstein series. 

This shifted convolution is obtained by taking the inner product of the function
\begin{align}
\overline{V_v(z)} &= y^\qtr\theta(z)E^*\lt(4z,\tfrac{1+v}{2}\rt) \label{vdef} \\
&\qquad\phantom{=}- 4^{\frac{1+v}{2}} \zeta^*(1+v) E_{\infty}^\hf \lt(z,\tfrac{3}{4} + \tfrac{v}{2}\rt) - \zeta^*(1+v) E_{0}^\hf (z,\tfrac{3}{4} + \tfrac{v}{2}) \notag
\end{align}
 with the Poincar\'{e} series alluded to in the introduction. By subtracting the Eisenstein series in \eqref{vdef} we ensure that $V_v(z)$ is indeed in $L^2(\Gamma \backslash \H)$. Let
\begin{equation}\label{poincare}
P_{h,Y}^{-\hf}(z,s;\delta) =\!\!\! \sum_{\gamma \in \Gamma_\infty \backslash \Gamma} \Im(\gamma z)^s \Psi_Y(\Im\gamma z) e^{-2\pi i h \Re(\gamma z)}e^{2\pi h \Im (\gamma z) (1-\delta) } \frac{j(\gamma,z)}{|j(\gamma,z)|}
\end{equation}
be the aforementioned Poincar\'{e} series, where  $\Psi_Y$ is the characteristic function for the interval $[Y\inv,Y]$. Again, note for brevity that we will often denote $P_{h,Y}^{-\hf}(z,s;\delta)$ as simply $P$. This Poincar\'{e} series has been constructed to pick out the $h$\th\ term from the Fourier expansion of an automorphic function of weight $\hf$. The factor of $j(\gamma,z)$s in the sum make $P$ itself, as a function of $z$, an automorphic function of weight  $-\hf$. Without $\Psi_Y$, the series as a function of $z$ would be growing exponentially in the imaginary part $y$ of $z$, which would have made it difficult to exploit the spectral decomposition of the function $P$. Right now, it is compactly supported, and hence in $L^2(\Gamma \backslash \H, -\hf)$.


\section{The Dirichlet Series}

In this section we will demonstrate that the inner product of $P$ and first term of $V_v$ produces a Dirichlet series of the form we desire.  As in the Rankin-Selberg convolution, we begin by unfolding the inner product given in \eqref{innerp}:
\begin{align}
\notag &\lt\langle P, \overline{\Im(*)^\qtr\theta E^*\lt(4*,\tfrac{1+v}{2}\rt)} \rt\rangle
\notag = \iint\limits_{Y^{\!-1} \,  0}^{\ \ \   Y  \ \ 1} y^{s-\frac{3}{4}}e^{-2\pi i h x}e^{2\pi h y (1-\delta)} \theta(z) E^*\lt(4z,\tfrac{1+v}{2}\rt)\frac{\mathrm{d}x \mathrm{d}y}{y} \displaybreak[0] .
\end{align}
Substituting the Fourier expansions for $\theta(z)$ and the Eisenstein series in the inegrand, the inner product becomes
\begin{align}
& \notag  \iint\limits_{Y^{\!-1} \,  0}^{\ \ \   Y  \ \ 1} y^{s-\frac{3}{4}}e^{-2\pi i h x}e^{2\pi h y (1-\delta)}\lt( \sum_{n=0}^\infty \tau(n)e^{2\pi i n z} \rt) \lt( \sum_{n=0}^\infty \overline{a_v(n,4y)}e^{-2\pi i n(4 x)} \rt) \tfrac{\dx \d y}{y} \displaybreak[0] ,
\end{align}
where 
$$
a_v(n,y) = \lt\{ \begin{array}{ll} \zeta^*(1+v)y^{\hf+\frac{v}{2}}+\zeta^*(v)y^{\hf-\frac{v}{2}} & \mbox{ if } n=0 \\
 2y^\hf\sigma_{-v}(n)|n|^{\frac{v}{2}} K_{\frac{v}{2}}(2\pi |n| y) & \mbox{ if } n \neq 0,
\end{array}
\rt.
$$
and $K_v(y)$ is the classical $K$-Bessel function.
 
Therefore we obtain the expansion
\begin{subequations} \label{unfold}
\begin{align}
\notag &\lt\langle P, \overline{\Im(*)^\qtr\theta E^*\lt(4*,\tfrac{1+v}{2}\rt)} \rt\rangle \\
&  =\tau(h)4^{\frac{3}{4}-s}\int_{Y^{-1}}^Y e^{-2\pi y \delta h} (4y)^{s-\qtr}\lt(\zeta^*(1+v)(4y)^{\frac{v}{2}}+\zeta^*(v)(4y)^{-\frac{v}{2}}\rt) \dy \label{unfold1} \\
& \ + 4\sum_{m=1}^{\lfloor \frac{h}{4}\rfloor} \tau(h-4m)\sigma_{-v}(m)|m|^{\frac{v}{2}} \int_{Y^{-1}}^Y y^{s-\qtr} e^{2\pi y(4m-\delta h )} K_{\frac{v}{2}}(2\pi |4m| y) \dy \label{unfold2} \\
&\label{dirichlet1} \  + 4\sum_{m=1}^{\infty} \tau(h+4m)\sigma_{-v}(m)|m|^{\frac{v}{2}} \mkern-9mu\int_{Y^{-1}}^Y \! y^{s-\qtr} e^{-2\pi y(4m+\delta h )} K_{\frac{v}{2}}(2\pi |4m| y) \dy. 
\end{align}
\end{subequations}
 We see that by taking the respective limits of $Y \to \infty$ and $\delta \to 0$ when $\Re(s)$ is sufficiently large,  \eqref{dirichlet1} becomes
 \begin{align}
 \frac{(16\pi)^{\frac{3}{4}-s}\Gamma(s-\qtr+\frac{v}{2})\Gamma(s-\qtr-\frac{v}{2})}{\Gamma(s+\qtr)}  \sum_{a,c=1}^\infty \frac{\tau(h+4ac)}{a^{s+\frac{v}{2}-\frac{1}{4}}c^{s-\frac{v}{2}-\qtr} }, \label{DirichletSeries}
 \end{align}
which we see is the multiple Dirichlet series we want to study.

Now, when $\Re s > \qtr+|\Re(\tfrac{v}{2})|$ we can take the limit as $Y$ goes to infinity and \eqref{unfold1} becomes
\begin{align*}
\tau(h)4^{\frac{3}{4}-s} \left((\tfrac{\delta h \pi}{2})^{\frac{1}{4} - s - \frac{v}{2}} \zeta^*(1\!+\! v) \Gamma(s\!-\!\tfrac{1}{4}\!+\!\tfrac{v}{2}) + (\tfrac{\delta h \pi}{2})^{\frac{1}{4} - s + \frac{v}{2}} \zeta^*(v) \Gamma(s\!-\!\tfrac{1}{4}\!-\!\tfrac{v}{2}) \right),
\end{align*}
which has a meromorphic continuation to all $s \in \mt{C}$. Similarly, when $\Re s > \qtr+|\Re(\tfrac{v}{2})|$, \eqref{unfold2} becomes
\begin{align*}
& 4\sum_{m=1}^{\lfloor \frac{h}{4}\rfloor} \frac{\tau(h-4m)\sigma_{-v}(m)|m|^{\frac{v}{2}}}{(8\pi m)^{s-\frac{1}{4}}} \int_0^\infty y^{s-\frac{1}{4}} e^{y(1- \frac{\delta h }{4m})} K_{\frac{v}{2}}(y) \dy \\
 & \ = 4\sum_{m=1}^{\lfloor \frac{h}{4}\rfloor} \frac{\tau(h-4m)\sigma_{-v}(m)|m|^{\frac{v}{2}}}{(8\pi m)^{s-\frac{1}{4}}} \sqrt{\frac{\pi}{2}}M_0(s+\tfrac{1}{4}, \tfrac{v}{2i}, \tfrac{\delta h}{4m}),
\end{align*}
as $Y \to \infty$. The function $M_k(s, z/i, \delta)$ has a meromorphic continuation to all $(s,z) \in \mt{C}^2$ for fixed $k \in \mt{R}$; its general definition and other relevant properties are due to \cite{Tom} and are summarized in Proposition \ref{props} in the appendix. Both \eqref{unfold1} and \eqref{unfold2} contribute poles at $s = \frac{1}{4} \pm \frac{v}{2}$. 

In order to understand the inner product of $P$ with $V_v$, where $V_v$ is as described in \eqref{vdef}, we also need to examine the inner product of $P$ with the half-integral weight Eisenstein series at the infinity cusp, 
\begin{align}
\notag &\lim_{Y \rightarrow \infty} \langle P, \overline{\zeta^*(1+v) E^\hf_{\infty}(*,\tfrac{3}{4} + \tfrac{v}{2})} \rangle\\
&\notag = \int_0^\infty \int_0^1 y^{s-1} e^{2\pi hy(1-\delta)} e^{-2\pi ihx} \zeta^*(1+v) E^\hf_{\infty}(z,\tfrac{3}{4} + \tfrac{v}{2}) \frac{\,dx \,dy}{y} \\
& = i^{\hf} \frac{\pi^{\frac{3}{4} +\frac{v}{2}} h^{\frac{v}{2}-\frac{1}{4}}}{\Gamma(1+\tfrac{v}{2})} \zeta^*(1+v) D_\infty(\tfrac{3}{4} + \tfrac{v}{2};h) (2\pi h)^{1-s} M_\hf (s,\tfrac{1}{i}(\tfrac{1}{4} + \tfrac{v}{2}), \delta),
\end{align}
and at the zero cusp,
\begin{align}
\notag &\lim_{Y \rightarrow \infty} \langle P, \overline{\zeta^*(1+v) E^\hf_{0}(*,\tfrac{3}{4} + \tfrac{v}{2})} \rangle\\
&\notag = \int_0^\infty \int_0^1 y^{s-1} e^{2\pi hy(1-\delta)} e^{-2\pi ihx} \zeta^*(1+v) E^\hf_{0}(z,\tfrac{3}{4} + \tfrac{v}{2}) \frac{\,dx \,dy}{y} \\
& = i^{\hf} \frac{\pi^{\frac{3}{4} +\frac{v}{2}} h^{\frac{v}{2}-\frac{1}{4}}}{\Gamma(1+\tfrac{v}{2})} \zeta^*(1+v) D_0(\tfrac{3}{4} + \tfrac{v}{2};h) (2\pi h)^{1-s} M_\hf (s,\tfrac{1}{i}(\tfrac{1}{4} + \tfrac{v}{2}), \delta).
\end{align}
Here 
\beq
D_{\infty}(s;h)= \sum_{n=1}^\infty \frac{g_h(4n)}{(4n)^s}
\eeq
is the $h$\th\ Fourier coefficient of $E^\hf_\infty(z,s)$, where $g_h(c)$ denotes the Gauss sum,
\beq
g_h(c) = \sum_{\substack{d \ (c) \\ d \mbox{  \tiny odd}}} \varepsilon_d \lt(\frac{c}{d} \rt) e^{\frac{2\pi i d h}{c}}.
\eeq
 Similarly $D_{\mathfrak{a}}(s;h)$ is the $h$\th\ Fourier coefficient of $E^\hf_\mathfrak{a}(z,s)$ and only differs from $D_\infty(s;h)$ in the $2$-place. This is described more fully in \cite{goldfeld1985eisenstein}.  In particular, in Corollary 1.3 of \cite{goldfeld1985eisenstein}, it is shown that 
\begin{equation}\label{eq:FourierCoefficientEisenstein}
	D_\mathfrak{a}(s;h)=L^*(2s-\tfrac12, \lt(\tfrac{4h}{\cdot}\rt)) F_\mathfrak{a}(s;h)
\end{equation}
where $F_\mathfrak{a}(s;h)$ is a function which is analytic for $\Re(s)\geq\hf$.  

So to summarize, we have proven the following proposition.
\begin{prop} \label{summaryprop1} The function
\beq
\mathfrak{D}_v(s;h,\delta):=\frac{4}{(8\pi)^{s-\qtr}}\sum_{a,c=1}^{\infty}  \frac{\tau(4ac+h)}{a^{s+\frac{v}{2}-\frac{1}{4}}c^{s-\frac{v}{2}-\frac{1}{4}}}  \int_0^\infty \! y^{s-\qtr} e^{- y(1+\frac{\delta h}{4m} )} K_{\frac{v}{2}}( y) \dy \label{DY}
\eeq
is absolutely convergent for $\Re s > \frac{5}{4}+|\Re(\frac{v}{2})|$ and has the expansion
\begin{subequations}
\begin{align}
\mathfrak{D}_v&(s;h,\delta) \\
& =\lim_{Y \to \infty} \lt\langle P_{h,Y}(*,s;\delta), V_v \rt\rangle  \label{Yapprox}  \\
& \ \  - 4\sum_{m=1}^{\lfloor \frac{h}{4}\rfloor} \frac{\tau(h-4m)\sigma_{-v}(m)m^{\frac{v}{2}}}{(8\pi m)^{s-\frac{1}{4}}} \sqrt{\frac{\pi}{2}}M_0(s+\tfrac{1}{4}, \tfrac{v}{2i}, \tfrac{\delta h}{4m}) \label{Yc} \\  
& \ \ - \frac{\tau(h)(\tfrac{\delta h \pi}{2})^{\frac{1}{4} - s-\frac{v}{2}} }{4^{s-\frac{3}{4}}} \left(\zeta^*(1\!+\! v) \Gamma(s\!-\!\tfrac{1}{4}\!+\!\tfrac{v}{2}) + (\tfrac{\delta h \pi}{2})^{v} \zeta^*(v) \Gamma(s\!-\!\tfrac{1}{4}\!-\!\tfrac{v}{2}) \right)\label{Yd} \\
& \ \ +  i^{\hf} \frac{\pi^{\frac{3}{4} +\frac{v}{2}} h^{\frac{v}{2}-\frac{1}{4}}}{\Gamma(1+\tfrac{v}{2})} \zeta^*(1+v) D_\infty(\tfrac{3}{4} + \tfrac{v}{2};h) (2\pi h)^{1-s} M_\hf (s,\tfrac{1}{i}(\tfrac{1}{4} + \tfrac{v}{2}), \delta) \label{Ye} \\
& \ \ +  i^{\hf} \frac{\pi^{\frac{3}{4} +\frac{v}{2}} h^{\frac{v}{2}-\frac{1}{4}}}{\Gamma(1+\tfrac{v}{2})} \zeta^*(1+v) D_0(\tfrac{3}{4} + \tfrac{v}{2};h) (2\pi h)^{1-s} M_\hf (s,\tfrac{1}{i}(\tfrac{1}{4} + \tfrac{v}{2}), \delta). \label{Yf}
\end{align}
\end{subequations}
in this region.
\end{prop}
Since \eqref{DY} converges to \eqref{DirichletSeries} as $\delta \to 0$, in the next section we will use this proposition and the inverse Mellin transform of \eqref{Yapprox} to relate the asymptotic expansion of the smoothed sum in Theorem \ref{boxtheorem} to the analytic properties of $\lt\langle P, V_v \rt\rangle$. 


\section{The inverse Mellin integrals}\label{4section}
To get an asymptotic estimate of
\beq
\sum_{a,c=1}^\infty \tau(h+4ac)e^{-(\frac{a+c}{X})}, \label{asymptotic}
\eeq
we take an inverse Mellin transform of the equation in Proposition \ref{summaryprop1}. The problem will be reduced to analyzing the inner product in \eqref{Yapprox}, which is accomplished in the next section via its spectral expansion.

\begin{prop}
For large $X \gg 1$, we have that
\beq
 \sum_{a,c=1}^\infty \tau(h + 4ac) e^{-\frac{a+c}{X}}=I - cX\log X - dX + O(X^{\hf+\vep}) \label{asymptote1}
\eeq
where $c$ and $d$ are constants that depend on $h$, with $c\neq 0$ if $\tau(h)\neq0$ and $d \neq 0$ if $\tau(h)=0$,  and $I$ is defined as
\beq
I:= \left(\frac{1}{2\pi i}\right)^2 \iint\limits_{(\frac 32)(3)} \lim_{Y \to \infty} \langle P, V_{w_1-w_2} \rangle \Gamma \lt(\tfrac{w_1+w_2+1}{2}\rt) (16\pi)^{\frac{w_1+w_2-1}{2}} X^{w_1+w_2} \,dw_1 \,dw_2. \label{mellin}
 \eeq

\end{prop}

\begin{proof}
Let $w_1 = s-\qtr+\frac{v}{2}$ and $w_2 = s-\qtr-\frac{v}{2}$. 
From \eqref{DirichletSeries}, we have that
\begin{align}
  &\sum_{a,c=1}^\infty \tau(h + 4ac) e^{-\frac{a+c}{X}} \notag \\
  =&\lim_{\delta \to 0} \left(\frac{1}{2\pi i}\right)^2 \mkern-9mu \iint\limits_{(\frac 32)(3)}\mathfrak{D}_{w_1-w_2}(\tfrac{w_1+w_2+\hf}{2};h,\delta) \frac{\Gamma \lt(\frac{w_1+w_2+1}{2}\rt)}{(16\pi)^{\frac{1-w_1-w_2}{2}}}  X^{w_1+w_2} \,dw_1 \,dw_2.  \label{mellin2}
\end{align}
Using the expansion of $\mathfrak{D}_{w_1-w_2}(\frac{w_1+w_2+\hf}{2};h,\delta)$ given in Proposition \ref{summaryprop1} we have that 
\beq
\sum_{a,c=1}^\infty \tau(h + 4ac) e^{-\frac{a+c}{X}} = \lim_{\delta \to 0}(I-T_1-T_2+T_{3,\infty}+T_{3,0})
\eeq
where 
\begin{align*}
  T_1 :=&\left(\frac{1}{2\pi i}\right)^2 \mkern-9mu \iint\limits_{(\frac 32)(3)}\! \frac{\tau(h)(4\pi)^{\frac{w_1+w_2-1}{2}}}{ (\frac{\delta h \pi}{2})^{w_1}} \zeta^*(1+w_1-w_2) \Gamma(w_1) \\
    &\mkern350mu \times \Gamma  \lt(\tfrac{w_1+w_2+1}{2} \rt) \! X^{w_1+w_2} \, dw_1 \,dw_2 \\
  & + \left(\frac{1}{2\pi i}\right)^2 \mkern-9mu \iint\limits_{(\frac 32)(3)} \frac{\tau(h)(4\pi)^{\frac{w_1+w_2-1}{2}}}{ (\frac{\delta h \pi}{2})^{ w_2}} \zeta^*(w_1-w_2) \Gamma(w_2) \\
    &\mkern350mu \times \Gamma \lt(\tfrac{w_1+w_2+1}{2} \rt) X^{w_1+w_2} \,dw_1 \,dw_2
\end{align*}
is due to \eqref{Yd},
\begin{align*}
 T_2 \! :=\! \! \left(\frac{1}{2\pi i}\right)^{\! 2} \iint\limits_{(\frac 32)(3)} &\frac{\Gamma \! \lt(\tfrac{w_1+w_2+1}{2} \rt) \!}{2^{\frac{1-w_1-w_2}{2}}} \! \sum_{m=1}^{\lfloor \frac{h}{4}\rfloor} \! \frac{\tau(h-4m)\sigma_{w_2-w_1}(m)}{m^{w_2}} \\
   &\mkern150mu \times M_0(\tfrac{w_1+w_2+1}{2}, \tfrac{w_1-w_2}{2i}, \tfrac{\delta h}{4m})  X^{w_1+w_2} dw_1 \,dw_2
\end{align*}
is due to \eqref{Yc}, and 
\begin{align}
  &T_{3,\mathfrak{a}}:= \\
  &\left(\frac{1}{2\pi i}\right)^2 \mkern-9mu \iint\limits_{(\frac 32)(3)}  i^{\hf} \frac{\pi^{\frac{3}{4} +\frac{w_1-w_2}{2}} h^{\frac{w_1-w_2}{2}-\frac{1}{4}}}{\Gamma(1+\tfrac{w_1-w_2}{2})} \zeta^*(1+w_1-w_2) D_\mathfrak{a}(\tfrac{3}{4} + \tfrac{w_1-w_2}{2};h) \Gamma(\tfrac{w_1+w_2+1}{2})  \notag 
  \\ & \mkern 80 mu
   \times \frac{(2\pi h)^{1-\frac{w_1+w_2+\hf}{2}}}{(16\pi)^{\frac{1-w_1-w_2}{2}}} M_\hf (\tfrac{w_1+w_2+\hf}{2},\tfrac{1}{i}(\tfrac{1}{4} + \tfrac{w_1-w_2}{2}), \delta) X^{w_1+w_2} \,dw_1 \,dw_2 \notag
\end{align}
is due to \eqref{Ye} and \eqref{Yf}. 

For $T_1$, we will deal with the two terms separately. For the first integral of $T_1$, we shift the line of integration to $\Re(w_2)=-\hf$ without passing over poles. We can then shift the line of integration to $\Re(w_1) = -\vep$, passing the simple pole at $w_1 = 0$. The residue at $w_1=0$ is
\begin{align*}
\frac{1}{2\pi i} \int\limits_{(-\frac 12)} \tau(h)(4\pi)^{\frac{w_2-1}{2}} \zeta^*(1-w_2) \Gamma(\tfrac{w_2+1}{2}) X^{w_2} \,dw_2,
\end{align*}
which is $O(X^{-\hf})$. We can take $\delta$ to $0$ in the shifted integral, effectively dropping the term.

For the other term of $T_1$, we will move the line of integration for $w_2$ left to $\Re(w_2)=-\hf$, picking up the residue at $w_2=0$. The shifted integral term can be dropped when $\delta$ is taken to $0$ and the residue at $w_2 = 0$ is
\[ \frac{1}{2\pi i} \int_{(3)} \tau(h)(4\pi)^{\frac{w_1-1}2} \zeta^*(w_1) \Gamma \lt(\tfrac {w_1+1}2 \rt) X^{w_1} \,dw_1 = \tau(h) X + O(X^\hf), \]
This equality is obtained by moving the line of integration to $\Re(w_1) = \hf$.

For $T_2$, moving lines of integration for $w_1$ left to $\Re(w_1) = \vep$ and for $w_2$ left to $\Re(w_2) = \vep$, we can take $\delta$ to $0$ and the resulting expression is $O(X^{\vep})$.

For $T_{3,\mathfrak{a}}$ we move the line of integration for $w_2$ to $\Re w_2 = \varepsilon$, encountering the simple pole at $w_2=\hf$. Shifting line of integration of $w_1$ to $\Re w_1 = \hf+\vep$ allows us to take $\delta$ to $0$ and shows that the moved integral is $O(X^{\hf+\vep})$. The residue at $w_2 = \hf$ is
\begin{align*}
  \frac{1}{2\pi i} \int\limits_{(3)} \frac{2i^{\hf}\sqrt{\pi}}{\Gamma(\tfrac{3}{4}+\frac{w_1}{2})} \zeta(\tfrac12+w_1) D_\mathfrak{a}(\tfrac12 + \tfrac{w_1}{2};h) X^{w_1+\hf} \,dw_1.
\end{align*}
Shifting the line of integration of the above to $\Re w_1 = \vep$, we pick up a pole at $w_1=\hf$. If $h$ is not a square, then the pole at $w_1=\hf$ is simple and comes from the Riemann zeta function and we have
\[ \lim_{\delta \to 0} (T_{3,\infty} + T_{3,0}) = c_2 X + O(X^{\hf+\vep}), \]
for some nonzero, computable constant $c_2$. If $h$ is a square, then $D_{\mathfrak{a}}$ also has a pole (see \eqref{eq:FourierCoefficientEisenstein}), and the pole at $w_1=\hf$ is a double pole instead of a simple one, and  we obtain:
\[ \lim_{\delta \to 0} (T_{3,\infty} + T_{3,0}) = c_3 X\log X + c_4 X + O(X^{\hf+\vep}) \]
where in the above, $c_3 \neq 0$ and $c_4$ are computable constants.
\end{proof}


\section{Spectral Expansion} \label{spectral}

The function $P_{h,Y}^{-\hf} (z,s;\delta) $ is in $L^2(\Gamma_0(4)\backslash \H, -\hf)$; in fact it is compactly supported, so we can expand this function in the spectrum of the Laplacian $\Delta_{-\hf}$. Let the $u_j$ be an orthonormal basis of Maass forms i.e.\ eigenfunctions of the Laplacian which vanish at the cusps. We parametrize the eigenvalues of these Maass forms as $1/4 + t_j^2$.
Each Maass form has the Fourier expansion 
\[
u_j(z) = \sum_{n\neq 0} \rho_j(n) |n|^{-\hf} W_{-\frac{\operatorname{sgn}(n)}{4}, it_j}(4\pi |n| y) e^{2\pi i n x},
\]
where $K_{\kappa,\mu}(y)$ is the classical Whittaker function. 

The continuous spectrum of $L^2(\Gamma_0(4)\backslash \H, -\hf)$ is spanned by the Eisenstein series at various cusps, $E_{\mathfrak{a}}^{-\hf}(z,\hf+it)$, defined for $\Re{u}>1$ by 
\[
E_{\mathfrak{a}}^{-\hf}(z,u) = \sum_{\Gamma_\mathfrak{a} \backslash \Gamma_0(4)} (\Im \sigma_{\mathfrak{a}}\inv \gamma z)^u \frac{j(\gamma,z)}{|j(\gamma,z)|}.
\]
This Eisenstein series has a meromorphic continuation to the whole complex plane with a pole at $u = \frac 34$, and 
\[
\Res_{u = \frac 34} E^{-\hf}_\mathfrak{a}(z,u) = c_{\mathfrak{a}}\overline{\theta(z)} y^\qtr
\]
is the residual spectrum, where $c_\mathfrak{a}$ is some constant. From this we can take the spectral expansion of  $P$:
\begin{align}
P_{h,Y}^{-\hf}(z,s;\delta) 
&=\sum_{j} \langle P,u_j\rangle u_j (z) 
+ \langle P, \Im(*)^\qtr\overline{\theta} \rangle y^\qtr\theta(z)  \notag \\
&+ \frac{\mathcal{V}}{4\pi i} \sum_{\mathfrak{a}= \infty,0, \hf} \int_{(\hf)} \langle P,E_{ \mathfrak{a}}^{-\hf}(*,u) \rangle E_{\mathfrak{a}}^{-\hf}(z,u) \du \label{spectralP}.
\end{align}

Now we see what each of these inner products is
\begin{align}
\langle P, u_j\rangle &= \iint_{\Gamma_0(4)\backslash \h} P_{h,Y}^{-\hf}(z,s;\delta)\overline{u_j(z)}\dz \displaybreak[0] \label{spectralY} \\
&= \frac{\overline{\rho_j(-h)}}{h^\hf} \frac{1}{(2\pi h)^{s-1}}\int_{2\pi hY\inv}^{2\pi hY} y^{s-1} e^{y(1-\delta)} W_{\qtr, it_j}(2y) \dy\displaybreak[0] \notag \\
&= (2\pi)^{-(s-1)}\frac{\overline{\rho_j(-h)}}{h^{s-\hf}} M_{Y,h,\frac{1}{2}} (s,t_j,\delta),\notag
\end{align}
where the function $M_{Y,h,\frac{1}{2}}(s,t_j,\delta)$ is defined by context and is described in more detail in the appendix and in \cite{Tom}. When $\Re s > \hf+\max_{t_j} \!|\Im(t_j)|$, Proposition \ref{ybound} lets us take the limit as $Y \to \infty$ inside the sum over $u_j$ in \eqref{spectralP}, so that the $M_{Y,h,\frac{1}{2}}(s,t_j,\delta)$ become $M_{\hf}(s,t_j,\delta)$. 

We similarly compute 
\begin{align}
\langle P_{h,Y}^{-\hf}(*,s;\delta), \Im(*)^\qtr \overline{\theta}\rangle &= \iint_{\Gamma_\infty \backslash \Gamma_0(4)} P_{h,Y}^{-\hf}(z,s;\delta) \theta(z)  y^\qtr \dz \label{thetaspec} \\
& \notag  = \int_{Y^{-1}}^Y \int_0^1 y^{s-\frac{3}{4}}e^{-2\pi i h x}e^{2\pi h y (1-\delta)}\! \! \lt( \sum_{n=0}^\infty \tau(n)e^{2\pi i n z} \rt) \!  \d x\dy \\
& \notag = \frac{\tau(h)}{(2\pi \delta h)^{s-\frac{3}{4}}} \int_{2\pi h \delta Y^{-1}}^{2\pi h \delta Y} y^{s-\frac{3}{4}}e^{-y} \dy,
\end{align}
which, for $\Re(s)>3/4$, uniformly converges as $Y \to \infty$ to 
\[
 \frac{\tau(h)}{(2\pi \delta h)^{s-\frac{3}{4}}} \Gamma(s-\tfrac{3}{4}).
\]

We use the Fourier coefficients of Eisenstein series to compute the inner product of $P$ and the Eisenstein series. By untiling we get that
\begin{align*}
\langle P_{h,Y}^{-\hf}&(*,s;\delta), E_{\mathfrak{a}}^{-\hf}(*,u)\rangle =\int_{Y\inv}^Y \int_0^1 y^s e^{-2\pi h x} e^{2\pi hy(1-\delta)}\overline{E_{\mathfrak{a}}^{-\hf}(z,u)} \dz.
\end{align*}
Thus when $\Re s > \hf + |\Re u -\hf|$, we have that
\begin{align}
\lim_{Y \to \infty} & \langle P_{h,Y}^{-\hf}(*,s;\delta), E_{\mathfrak{a}}^{-\hf}(*,u)\rangle  \label{continuousY}\\
&=  i^{\hf} \frac{\pi^{u} h^{u-1}}{\Gamma(u+\qtr)} D_\mathfrak{a}(u;h) (2\pi h)^{1-s} M_\hf (s,\tfrac{1}{i}(u-\thf), \notag\delta),
\end{align}
and Proposition \ref{ybound} allows us to take this limit through the integral with respect to $u$ in \eqref{spectralP}.

The spectral expansion of the Poincar\'{e} series, given in \eqref{spectralP}, can now be used to complete our computation of the asymptotic expansion given in \eqref{asymptote1} by allowing us to compute
\[
I:=\lim_{\delta \to 0} \left(\frac{1}{2\pi i}\right)^2 \iint\limits_{(\frac 32)(3)} \lim_{Y \to \infty} \langle P, V_{w_1-w_2} \rangle \Gamma \lt(\tfrac{w_1+w_2+1}{2}\rt) (16\pi)^{\frac{w_1+w_2-1}{2}} X^{w_1+w_2} \,dw_1 \,dw_2.
 \]
Indeed, following from \eqref{spectralP} we have
\begin{align}\label{spectralExpansion}
\langle P,V_v\rangle =& \sum_j \langle P,u_j\rangle \langle u_j, V_v \rangle + \langle P, \Im(*)^{\qtr}\theta\rangle \langle  \Im(*)^{\qtr}\theta , V_v\rangle \\
&+ \frac{\mathcal{V}}{4\pi}\sum_{\mathfrak{a} =  \infty,0,\hf} \int_{-\infty}^\infty \langle P, E_{\mathfrak{a}}(*,\thf+it) \rangle \langle E_\mathfrak{a}(*,\thf+it), V_v\rangle \  \d t. \notag
\end{align}


Now we can combine all the computations for the inner products of $P$ with various eigenfunctions. Summarizing the results, we have 
\begin{align}
&\lim_{Y \to \infty} \langle P_{h,Y}(*,s,\delta), u_j\rangle = (2\pi)^{-(s-1)}\frac{\overline{\rho_j(-h)}}{h^{s-\hf}} M_{\frac{1}{2}} (s,t_j,\delta) \label{ywentpuj},\\
&\lim_{Y \to \infty } \langle P_{h,Y}(*,s,\delta) , \Im(*)^{\qtr}\theta \rangle =  
 \frac{\tau(h)}{(2\pi \delta h)^{s-\frac{3}{4}}} \Gamma(s-\tfrac{3}{4})\label{ywentTheta},
 \end{align}
 and
 \begin{align}
 &\lim_{Y\to \infty }\langle P_{h,Y}(*,s,\delta) , E_{\mathfrak{a}}(*,u) \rangle = \frac{ i^{\hf}\pi^{u} h^{u-1}}{\Gamma(u+\qtr)} D_\mathfrak{a}(u;h) (2\pi h)^{1-s} M_\hf (s,\tfrac{1}{i}(u-\thf), \delta).
\end{align} 
Since an Eisenstein series is a component of each of the summands of $V_v$,  we can compute $\langle u_j, V_v \rangle$ by general unfolding techniques to get that, for $\re v > \hf$,  
\begin{equation} \label{unfolded}
\langle u_j, V_v \rangle \ll_{\Re v}[(1+|t_j + \Im(\tfrac{v}{2})|)(1+|t_j - \Im(\tfrac{v}{2})|)]^{\frac{\Re v}{2} - \frac{1}{4}}e^{-\frac{\pi}{2}||t_j|-|\!\Im (\frac{v}{2})||}.
\end{equation} 
We can similarly treat with the other inner products of $V_v$ in \eqref{spectralExpansion}. Using Proskurin's generalization of the Kuznetsov trace formula, as referenced in \cite{Duke}, we can show that we have $\rho_j(-h) \ll_h e^{\frac{\pi}{2}|t_j|}(1+|t_j|)^\vep$ on average when we sum over $t_j$s. This along with line \eqref{prop1} of Proposition \ref{props} gives us that this series \eqref{spectralExpansion} converges locally normally as $Y \to \infty$ for all $s \in \mathbb{C}$ away from the poles. We note, however, that the presence of the $\delta^{-A}$ term in \eqref{prop1} prevents us from taking the limit as $\delta \to 0$.


We will now proceed to substitute the spectral expansion for $\langle P,V_v \rangle$ for $I$ in \eqref{mellin}, along with the variable substitutions $s = \frac{w_1+w_2}{2}+\qtr$ and $v = w_1 - w_2$, and from this we will obtain the following propositon. As a corollary, we get the smoothed sum asymptotic in Theorem \ref{boxtheorem}.

\begin{prop}
For the $I$ defined as in \eqref{mellin}, we have that
\[
I =\tau(h)cX+ O(X^{\hf}),
\]
for some computable constant $c$ which is independent of $h$.
\end{prop}

\begin{proof}
We begin by dealing with each component of the spectral expansion separately. By substituting \eqref{ywentpuj} into the cuspidal part of  \eqref{spectralExpansion}, we get the discrete part of the spectrum,
\begin{align*}
I_{\text{cusp}}&:= \lim_{\delta \to 0} \sum_{j} \left(\frac{1}{2\pi i}\right)^2 \mkern-9mu \iint\limits_{(\frac 32)(3)}\! \!\! \lim_{Y \to \infty}\!\! \langle P_{h,Y}(*,\tfrac{w_1+w_2+\hf}{2},\delta), u_j\rangle\langle u_j,\! V_{w_1-w_2}\rangle \\[-2ex]
&\mkern250mu \times  \Gamma \!\lt(\tfrac{w_1+w_2+1}{2}\rt) \! (16\pi)^{\frac{w_1+w_2-1}{2}} X^{w_1+w_2} \,dw_1 \,dw_2 \\
&=  \lim_{\delta \to 0}\sum_{j} \left(\frac{1}{2\pi i}\right)^2 \mkern-9mu \iint\limits_{(\frac 32)(3)}\frac{\pi^\qtr \overline{\rho_j(-h)}}{ 2^{\frac{5}{4}-\frac{3w_1+3w_2}{2}}h^{\frac{w_1+w_2-\hf}{2}}} M_{\frac{1}{2}} (\tfrac{w_1+w_2+\hf}{2},t_j,\delta) \langle u_j, V_{w_1-w_2}\rangle \\
&\mkern350mu \times \Gamma \lt(\tfrac{w_1+w_2+1}{2}\rt) 
 \! X^{w_1+w_2} \,dw_1 \,dw_2.
\end{align*}
To be able to take the limit as $\delta \to 0$, we need to shift $\Re(w_1+w_2)$ sufficiently far to the left so that \eqref{ll1} and \eqref{ll2} of Proposition \ref{props} of the appendix allow the sum over $t_j$ to converge independently of $\delta$. So moving the line of integration for $w_2$ to $\Re w_2 = -\frac{3}{4}$ and then doing the same for $w_1$ to $\Re w_1 = -\qtr + \vep$, we pass over simple poles at $w_1 = \hf \pm 2it_j - w_2$ and the shifted integral resulting from $I_{\mbox{\tiny cusp}}$ is $O(X^{-1+ \vep})$. This follows from the same reasoning that demonstrated the convergence of \eqref{spectralExpansion} as $Y \to \infty$.

For the residual terms, we use \eqref{resfin} of Proposition \ref{props} to compute the residues of $M_{\hf}(s,t_j,\delta)$. Letting $\delta \to 0$ we get these become,
\begin{align} \label{sumjs}
  &\sum_{\pm t_j} \left(\frac{1}{2\pi i}\right)^2 \mkern-9mu \int\limits_{(-\frac{3}{4}-2\vep)}\frac{\pi^\qtr \overline{\rho_j(-h)}}{ 2^{-2it_j}h^{it_j}} \frac{2\Gamma(2it_j)}{\Gamma(\qtr+it_j)} \langle u_j, V_{\hf+2it_j-2w_2}\rangle \Gamma \lt(\tfrac{3}{4}+it_j\rt) 
 \! X^{\hf+2it_j} \,dw_2.
\end{align}
The Shimura correspondence for Maass forms, as investigated by Sarnak and Goldfeld in \cite{Sarnak}, gives that each weight-$\hf$ Maass form on $\Gamma_0(4)$ and eigenvalue $t_j$ lifts to a weight-zero Maass form on $\Gamma_0(4)$ with eigenvalue $2t_j$. Since there are no exceptional eigenvalues for weight zero and level two, as is computed explicitly in \cite{LMFDB}, there are no exceptional eigenvalues for weight-$\hf$ and level four. Hence \eqref{sumjs} is $O(X^\hf)$.

 A similar approach is applicable to the continuous spectrum and yields the same overall bound of $O(X^\hf)$.

All that remains is the residual spectrum:
\begin{align*}
I_{\text{res}}&:=\lim_{\substack{Y \to \infty\\\delta \to 0}} \left(\frac{1}{2\pi i}\right)^2 \iint\limits_{(\frac 32)(3)} \langle P_{h,Y}(z,\tfrac{w_1+w_2+\hf}{2},\delta), \overline{\theta(z)}y^\qtr \rangle \langle \overline{\theta(z)}y^\qtr,V_{w_1-w_2} \rangle \\[-3ex]
&\bquad\aquad \times \Gamma \lt(\tfrac{w_1+w_2+1}{2}\rt) (16\pi)^{\frac{w_1+w_2-1}{2}} X^{w_1+w_2} \,dw_1 \,dw_2\\ 
&= \lim_{\delta \to 0} \left(\frac{1}{2\pi i}\right)^2 \iint\limits_{(\frac 32)(3)} \frac{\tau(h)}{(2\pi\delta h)^{\frac{w_1+w_2-1}{2}}} \Gamma(\tfrac{w_1+w_2-1}{2}) \langle \overline{\theta(z)}y^\qtr,V_{w_1-w_2} \rangle \\[-3ex]
&\bquad\aquad \times \Gamma \lt(\tfrac{w_1+w_2+1}{2}\rt) (16\pi)^{\frac{w_1+w_2-1}{2}} X^{w_1+w_2} \,dw_1 \,dw_2.
\end{align*}
We can shift the line of integration of $w_2$ to $\Re w_2 = \vep$ without passing over poles. Moving the line of integration of $w_1$ to $\Re w_1 = \hf $, past the pole at $w_1=1-w_2$, and taking $\delta \to 0$ we get that the shifted integral vanishes and the residual term is
\begin{align*}
I_{\text{res}}&= \frac{1}{2\pi i} \int\limits_{(\vep)} 2\tau(h) \langle \overline{\theta(z)}y^\qtr,V_{1-2w_2} \rangle \,dw_2 \cdot X =: c\tau(h)X.
\end{align*}
This completes the proof of the proposition.
\end{proof}

\section{Sums over hyperbolic regions}
In this section, we discuss how the methods given above can be modified to obtain  the results for the hyperbolic-shaped region of Theorem \ref{hyperbolictheorem}. Since the techniques used are very similar, we only explain the differences and omit technical details.
%
%

The target sum can be rewritten as
\[ \sum_{\substack{a,c \\ 1 \leq ac \leq X^2}} \tau(4ac + h) = \sum_{m \leq X^2} \tau(4m+h)\sigma_0(m). \]
This can be obtained via Perron's formula:
\[ \sum_{m \leq X^2} \tau(4m+h)\sigma_0(m) = \frac{1}{2\pi i} \int_{1+\vep-iT}^{1+\vep+iT} \sum_{m \geq 1} \frac{\tau(4m+h)\sigma_0(m)}{m^s} X^{2s} \,\frac{ds}{s} +  O\lt(\frac{X^{1+\vep}}{T}\rt). \]
The Dirichlet series above is extremely close to the one in \eqref{dseries} that we used for the smooth sum estimate, with the major difference that $v=0$. It is natural to attempt setting $v=0$. However, the troubles are that $V_v$ is not square-integrable at $v=0$ and that $V_v$ has poles at $v=0$. The fix to the problem is to replace the original $V_v$ with the following function,
\begin{align*}
  \overline{V_0} = &E^*(4z,\hf) \theta(z)y^{\qtr} \\
   &- \operatorname{Const}_{v=0} [4^{\frac{1+v}{2}} \zeta^*(1+v) E_\infty^\hf(z,\tfrac{3}{4} + \tfrac{v}{2}) - e^{-\frac{i\pi}{4}} \zeta^*(1+v) E_0^\hf(z,\tfrac{3}{4}+\tfrac{v}{2})] \\
   &- \operatorname{Const}_{v=0} [4^{\frac{1-v}{2}} \zeta^*(1-v) E_\infty^\hf(z,\tfrac{3}{4} - \tfrac{v}{2}) - e^{-\frac{i\pi}{4}} \zeta^*(1-v) E_0^\hf(z,\tfrac{3}{4}-\tfrac{v}{2})].
\end{align*}
From this point onwards, most calculations are similar to the smoothed sum case, barring some additional calculations necessary on differentiating $M_k(s,\frac{z}{i},\delta)$ with respect to the $z$ variable as well as bounding the corresponding $\mathfrak{D}_0(s;h)$ when $\Re s = \hf + \vep$. By choosing an optimal $T$, we are able to arrive at the error bound of $O(X^{\frac{4}{5}+\vep})$.


\begin{appendix}
\section{}
Consider the functions 
\beq
M_{Y,h,k}(s,z/i,\delta):=\int_{Y^{-1}2\pi h}^{Y2\pi h}  y^{s-1}e^{y(1-\delta)}W_{{\frac{k}{2}},z}(2y) \frac{dy}{y}.
\eeq 
and
\beq
M_{k}(s,z/i,\delta):=\int_{0}^\infty y^{s-1} e^{y(1-\delta)} W_{\frac{k}{2}, z}(2y) \dy
\eeq
for $s,z\in \C$, $k\in\R$, $Y \gg 1$, $h \in \mt{Z}_{\geq 1}$ and small $\delta>0$. These functions have been thoroughly investigated in \cite{Tom} by the first author, further generalizing a construction first studied in \cite{Jeff}.  The following two propositions summarize the relevant properties of these functions used in this work, and their proofs can be found in \cite{Tom}.
\begin{prop} \label{ybound}
Let
\beq
M_{Y,h,k}(s,z/i,\delta):=\int_{Y^{-1}2\pi h}^{Y2\pi h}  y^{s-1}e^{y(1-\delta)}W_{{\frac{k}{2}},z}(2y) \frac{dy}{y}.
\eeq 
For fixed $\vep>0, Y \gg 1, $ $1 > \delta >0$, and $A \in \mt{Z}_{\geq 0}$, we have that for $\Re(s)>\hf+|\Re (z)| +\vep$
\begin{align}
|M_{Y,h,k}(s,z/i,\delta)-&M_k(s,z/i,\delta)| \label{ygrowth}
 \\
 & \ll \frac{e^{-Y2\pi h \delta}(Yh)^{\Re s +\frac{k}{2}+A+\vep-2}}{\delta(1+|\Im z|)^A}+\frac{(Y^{-1}h)^{\Re s-\hf-|\Re z|-\vep}}{(1+|\Im z|)^A} \notag
 \end{align} 
 where the implied constant is dependent on $A$, $k$, $\Re(s)$, $\Re(z)$  and $\vep$ . 
\end{prop}

\begin{prop} \label{props} Fix small $\vep>0$ and $\delta>0,$  and let $k \in \mt{R}$. Furthermore let $s=\sigma+ir$ where $\sigma,r \in \mt{R}$ and $\Im(z) =t$. The function $M_k(s,z/i,\delta)$ has a meromorphic continuation to all $(s,z)\in\mt{C}^2$ with simple polar lines at the points $s-\hf\pm z\in \mt{Z}_{\leq 0}$ when $\hf - \kt\pm z \notin \mt{Z}_{\leq 0}$. For fixed $z\notin \hf \mt{Z}$, the residues at these points are given by
\bal
 \label{resfin} \ \ \ \stackrel[s=\hf-\ell\pm z]{}{\mbox{\emph{Res}} }M_k(s,z/i,\delta) =&\frac{(-1)^\ell2^{\hf+\ell\mp z}\Gamma(\hf \mp z-\kt+\ell)\Gamma(\pm 2z-\ell)}{\ell! \Gamma(\hf-\kt + z)\Gamma(\hf-\kt-z)}\\
& \notag +\mc{O}_{\ell,\re z}\lt((1+|t|)^{\ell+\kt-\hf-\re z -b}e^{-\frac{\pi}{2}|t|}\delta  \rt) 
\end{align}
where $\ell \in \mt{Z}_{\geq 0}$ and $b< \min(-1,\hf-\sigma-\Re(z),-2\Re(z))$. If $\ell\pm 2z \in \mt{Z}_{\geq 0}$ then $ M_k(s,z/i,\delta)$ has a double pole at $s=\hf-\ell \mp z$. 


  
  For $s$ and $z$ at least a distance of $\vep>0$ from the poles, for any $A  \in \mt{R}$, independent of $\delta$, $r$, and $t$, such that $A> 1+|\sigma| + |\Re(z)|+|\tkt|$ and
\beq
M_k(s,z/i,\delta) \ll_{A,\vep} \delta^{-A}(1+|t|)^{2\sigma - 2 - 2A+k}(1+|r|)^{9A}e^{-\frac{\pi}{2}|r|}.
\label{prop1}
\eeq
 For $\sigma < 1-\kt-\vep_0$ and $s$ at least a distance of $\vep$ away from the poles of $M_k(s,z/i,\delta)$ and $\delta(1+|t|)^2 \leq 1$ we have 
 \bal 
 \label{ll1} M_k(s,z/i,\delta)&=\frac{2^{1-s}\Gamma(s-\frac{1}{2}-z)\Gamma(s-\hf+z)\Gamma(1-s-\kt)}{\Gamma(\frac{1}{2}-\frac{k}{2}+z)\Gamma(\hf-\kt-z)} \\
& \notag + \mc{O}_{A,b,\vep_0} \lt( (1+|t|)^{2\sigma -2+k+2\ep}(1+|r|)^{9A-2b}e^{-\frac{\pi}{2}|r|}\delta^{\vep_0} \rt) 
 \end{align}
while for $\delta(1+|t|)^2>1$ we have
 \beq
M_k(s,z/i,\delta) \ll_{A,\vep} (1+|t|)^{2\sigma - 2 +k}(1+|r|)^{9A}e^{-\frac{\pi}{2}|r|}.
\label{ll2}
\eeq

\end{prop}
\end{appendix}

\bibliography{bibfile}	
\bibliographystyle{alpha}	
\end{document}